\documentclass[12pt]{amsart}
\usepackage[margin=1in]{geometry}                
\usepackage{graphicx}
\usepackage{amssymb}
\usepackage{amsthm}
\usepackage{mathrsfs}
\usepackage{caption}
\usepackage{xcolor}

\newtheorem{theorem}{Theorem}[section]
\newtheorem{proposition}[theorem]{Proposition}
\newtheorem{lemma}[theorem]{Lemma}
\newtheorem{definition}[theorem]{Definition}
\newtheorem{corollary}[theorem]{Corollary}

\numberwithin{equation}{section}

\newcommand{\ZZ}{\mathbb{Z}}
\newcommand{\RR}{\mathbb{R}}
\newcommand{\NN}{\mathbb{N}}

\newcommand{\bbA}{\mathbb{A}}
\newcommand{\bbB}{\mathbb{B}}

\newcommand{\cali}{\mathcal{I}}
\newcommand{\calj}{\mathcal{J}}
\newcommand{\calk}{\mathcal{K}}

\newcommand{\Div}{{\mathrm{Div}}}

\newcommand{\exch}{{\sf Ex}}


\title[Splitting for integer tilings]{Splitting for integer tilings} 
\author{Izabella {\L}aba
and Itay Londner
 }

\date{\today} 

\begin{document}

\begin{abstract}We consider translational integer tilings by finite sets $A\subset\ZZ$. We introduce a new method based on \emph{splitting}, together with a new combinatorial interpretation of some of the main tools of \cite{LaLo1}, \cite{LaLo2}. We also use splitting to prove the Coven-Meyerowitz conjecture for a new class of tilings 
$A\oplus B=\ZZ_M$. This includes tilings of period $M=p_1^{n_1}p_2^{n_2}p_3^{n_3}$ with $p_1>p_2^{n_2-1}p_3^{n_3-1}$, and tilings of period $M=p_1^{n_1}p_2^2p_3^2p_4^2$ with $p_1>p_2p_3p_4$, where $p_1,p_2,p_3,p_4$ are distinct primes and $n_1,n_2,n_3\in\NN$.

\end{abstract}

\maketitle

\setcounter{tocdepth}{1}

\tableofcontents


\section{Introduction}


Let $A\subset \ZZ$ be a finite and nonempty set. We say that 
$A$ tiles the integers by translations, and call it a {\it finite tile}, if there is a set $T\subset\ZZ$ such that each $n\in\ZZ$ has a unique representation
$n=a+t$ with $a\in A$ and $t\in T$. It is well known \cite{New} that any such tiling must be periodic, so that there exists a $M\in\NN$
such that $T=B\oplus M\ZZ$ for some finite set $B\subset \ZZ$. 
We then have $|A|\,|B|=M$, so that $A\oplus B=\ZZ_M$
is a factorization of the cyclic group $\ZZ_M$.

By translational invariance, we may assume that
$A,B\subset\{0,1,\dots\}$. The {\it mask polynomials} of $A$ and $B$ are
$$
A(X)=\sum_{a\in A}X^a,\ B(x)=\sum_{b\in B}X^b .
$$
In this language, $A\oplus B=\ZZ_M$ is equivalent to
\begin{equation}\label{poly-e1}
 A(X)B(X)=1+X+\dots+X^{M-1}\ \mod (X^M-1).
\end{equation}

Recall that the $s$-th cyclotomic polynomial $\Phi_s(X)$ is the minimal polynomial of $e^{2\pi i/s}$. The identity
$$
X^n-1=\prod_{s|n}\Phi_s(X),
$$
allows us to rephrase (\ref{poly-e1}) as
$$
  |A||B|=M\hbox{ and }\Phi_s(X)\ |\ A(X)B(X)\hbox{ for all }s|M,\ s\neq 1.
$$
Since $\Phi_s$ are irreducible, each $\Phi_s(X)$ with $s|M$ must divide at least one of $A(X)$ and $B(X)$.

The Coven-Meyerowitz theorem \cite{CM} is as follows. 

\begin{theorem}\label{CM-thm} \cite{CM}
Let $S_A$ be the set of prime powers
$p^\alpha$ such that $\Phi_{p^\alpha}(X)$ divides $A(X)$.  Consider the following conditions.

\smallskip
{\it (T1) $A(1)=\prod_{s\in S_A}\Phi_s(1)$,}

\smallskip
{\it (T2) if $s_1,\dots,s_k\in S_A$ are powers of different
primes, then $\Phi_{s_1\dots s_k}(X)$ divides $A(X)$.}
\medskip

Then:

\begin{itemize}

\item if $A$ satisfies (T1), (T2), then $A$ tiles $\ZZ$;

\item  if $A$ tiles $\ZZ$ then (T1) holds;

\item if $A$ tiles $\ZZ$ and $|A|$ has at most two distinct prime factors,
then (T2) holds.
\end{itemize}

\end{theorem}

We do not know whether all finite tiles satisfy (T2). The statement that this is true 
has become known in the literature as the {\it Coven-Meyerowitz conjecture} and is generally considered to be the main open question in the theory of integer tilings. The progress to date includes our work in \cite{LaLo1}, \cite{LaLo2}, leading to Theorem \ref{T2-3primes-thm} in the odd case and partial results for general tilings, as well as mild extensions of the Coven-Meyerowitz theorem in  \cite{M}, \cite{shi}, \cite{Tao-blog}, and proofs of (T2) for tilings of certain special types \cite{KL}, \cite{dutkay-kraus}. There is also significant interest in other tiling questions, for example translational tilings of $\ZZ^n$ and $\RR^n$ \cite{Bh}, \cite{GK}, \cite{GT}, \cite{GT2}, and tilings of the real line by a function \cite{KoLev}).

The main result of this paper is the following theorem. For $N\in \NN$, we will use the notation 
$$D(N)=\frac{N}{\prod_{p|N,p\ {\rm prime}} p}.$$

\begin{theorem}\label{T2-largeprime} 
Let $M=p_1^{n_1}M_1$, where $p_1$ is prime, $n_1\in\NN$, and $p_1\nmid M_1$.
Assume that for any $M'|M_1$, (T2) holds for both $A'$ and $B'$ in any tiling $A'\oplus B'=\ZZ_{M'}$, and
that
\begin{equation}\label{e-largeprime}
p_1>D(M_1).
\end{equation}
Then in any tiling $A\oplus B=\ZZ_M$, both $A$ and $B$ satisfy (T2).
\end{theorem}

Theorem \ref{T2-largeprime} is a natural application of {\it splitting}, a new method developed in this article and described briefly below. The proof is much shorter and simpler than those in \cite{LaLo1}, \cite{LaLo2}. Splitting is also an important part of the proof of Theorem \ref{T2-3primes-thm}.
In \cite{LaLo2}, we proved Theorem \ref{T2-3primes-thm} under the additional assumption that $M$ is odd. 
In the companion paper \cite{LL-even} we use a combination of splitting and the methods of \cite{LaLo1}, \cite{LaLo2}
to extend the same result to the even case, thus completing the proof of the theorem.

\begin{theorem}\label{T2-3primes-thm}\cite{LaLo2}, \cite{LL-even}
Let $M=p_1^2p_2^2p_3^2$, where $p_1,p_2,p_3$ are distinct primes. 
Assume that $A\oplus B=\ZZ_M$, with $|A|=|B|=p_1p_2p_3$. Then both $A$ and $B$ satisfy (T2).
\end{theorem}

In order to apply Theorem \ref{T2-largeprime}, we need to have a number $M'$ such that (T2) holds for both tiles in any $M'$-periodic tiling. Such choices of $M'$ are provided by Theorems \ref{CM-thm}, \ref{T2-3primes-thm}, and \cite[Corollary 6.2]{LaLo1}.
Corollary \ref{manyprimes} below states the most general assumptions on $M$ under which our work here, combined with \cite{LaLo1}, \cite{LaLo2}, \cite{LL-even}, yields (T2) for both tiles in any tiling $A\oplus B=\ZZ_M$.

\begin{corollary}\label{manyprimes}
Let 
$$M^*=p_1^{\alpha_1}p_2^{\alpha_2}p_3^{\alpha_3}p_4p_5\dots p_K q_1^{\beta_1}\dots q_L^{\beta_L},
$$
where $p_1,\dots,p_K,q_1,\dots,q_L$ are distinct primes, $\beta_1,\dots,\beta_L\in\NN$, and either
\begin{equation}\label{alpha-e1}
\alpha_3=1,\ \alpha_1,\alpha_2\in\NN,
\end{equation}
or
$$
 \alpha_1=\alpha_2=\alpha_3=2.
$$
Assume further that 
\begin{equation}\label{growth-e1}
q_j>D(N_{j})\hbox{ for }j=1,\dots,L,
\end{equation}
where $N_1:=p_1^{\alpha_1}p_2^{\alpha_2}p_3^{\alpha_3}$ and 
$$
N_j:=N_{j-1} q_j^{\beta_j}
\hbox{ for }j=2,\dots,L.
$$
Then for any $M|M^*$, both tiles $A$ and $B$ in any tiling $A\oplus B=\ZZ_{M}$ satisfy (T2).
\end{corollary}

The two special cases below use 
Theorem \ref{T2-largeprime} together with Theorems \ref{CM-thm} and \ref{T2-3primes-thm}.

\begin{corollary}\label{largeprime-cor} 
Assume that $A\oplus B=\ZZ_M$, and that one of the following holds:
\begin{itemize}
\item $M=p_1^{n_1}p_2^{n_2}p_3^{n_3}$ with $p_1>p_2^{n_2-1}p_3^{n_3-1}$,
\item $M=p_1^{n_1}p_2^2p_3^2p_4^2$ with $p_1>p_2p_3p_4$,
\end{itemize}
where $p_1,p_2,p_3,p_4$ are distinct primes and $n_1,n_2,n_3\in\NN$.
Then both $A$ and $B$ satisfy (T2).
\end{corollary}

We note the application to Fuglede's conjecture \cite{Fug} which (in its original formulation) states that a set $\Omega\subset \RR^n$ of positive $n$-dimensional Lebesgue measure tiles $\RR^n$ by translations if and only if the space $L^2(\Omega)$ admits an orthogonal basis of exponential functions. A set with the latter property is called {\em spectral}. 
The conjecture is known to be false, in its full generality, in dimensions 3 and higher \cite{tao-fuglede}, \cite{KM}, \cite{KM2}, \cite{FMM}, \cite{matolcsi}, \cite{FR}. However, there are important special cases in which the conjecture was confirmed \cite{IKT}, \cite{GL}, \cite{LM}, and the finite abelian group analogue of the conjecture is currently a very active area of research, see e.g.~\cite{FKMS}, \cite{IMP}, \cite{KMSV2}, \cite{M}, \cite{MK},  \cite{shi}.

In dimension 1, the problem is still open in both directions, and the ``tiling implies spectrum" direction hinges on proving (T2) for all finite tiles (\cite{LW1}, \cite{LW2}, \cite{L}; see also \cite{dutkay-lai} for an overview of the problem and an investigation of the converse direction). Our Corollary \ref{manyprimes}, combined with \cite[Theorem 1.5]{L} and \cite[Corollary 6.2]{LaLo1}, yields the following result.

\begin{corollary}\label{cor-Fuglede}
Let $M$ satisfy the assumptions of Corollary \ref{manyprimes}. Then:

\begin{itemize}
\item[(i)] If $A\subset\ZZ_M$ tiles $\ZZ_M$ by translations, then it is spectral.

\item[(ii)] Let $A\subset \ZZ$ be a finite set such that $A$ mod $M$ tiles $\ZZ_{M}$, and let ${F=\bigcup_{a\in A}[a,a+1]}$, so that $F$ tiles $\RR$ by translations. Then $F$ is spectral. 
\end{itemize}

\end{corollary}

As mentioned above, our arguments are based on \emph{splitting} (Section \ref{splitting-section}). Given a fiber in $\ZZ_M$ (an arithmetic progression of step $M/p$ and length $p$ for some prime divisor $p|M$), we ask which elements of $A$ and $B$ tile the elements of that progression. It is not difficult to prove that, for each fiber separately, these elements must follow a certain ``splitting'' pattern. If these splitting patterns are uniform for all $M$-fibers corresponding to some prime factor $p$, and if this uniformity persists under dilations of the sets $A$ and $B$, we are able to apply a tiling reduction developed in \cite[Section 6.2]{LaLo1}. 

If the assumption (\ref{e-largeprime}) holds, it is relatively easy both to prove such uniformity and to set up an inductive argument proving (T2). For other classes of tilings, further work is needed. Towards this end, in Section \ref{consistency} we prove several partial results on splitting consistency on grids that are needed in the proof of Theorem \ref{T2-3primes-thm} in \cite{LL-even}.

An additional goal of this paper is to revisit and simplify some of our results from \cite{LaLo1, LaLo2}. In Section \ref{product-dilations}, we provide 
a simple proof of Theorem \ref{GLW-thm} based on the dilation theorems of Sands and Tijdeman, and a new combinatorial interpretation of concepts (such as saturating sets) derived from that theorem. In Section \ref{tiling-reductions}, we give a new formulation of the {\it slab reduction} from \cite{LaLo1} in terms of splitting. This offers an easier way to verify the criteria of the reduction (see \cite[Section 9]{LL-even} for examples), and supplies a key part of our proof of Theorem \ref{T2-largeprime}.


\section{Notation and preliminaries}\label{sec-prelim}


We assume that $M=p_1^{n_1}\dots p_K^{n_K}$, where $p_1,\dots,p_K$ are distinct primes and $n_1,\dots,n_K\in\NN$. If $A\subset\ZZ_M$, its mask polynomial is $A(X)=\sum_{a\in A}X^a$.
We use the convolution notation $A*B$ for the weighted sumset of $A$ and $B$ (interpreted as a multiset if necessary), so that $(A*B)(X)=A(X)B(X)$.
If one of the sets is a singleton, say $A=\{x\}$, we write $x*B=\{x\}*B$. 
The direct sum notation $A\oplus B$ is reserved for tilings, so that $A\oplus B=\ZZ_M$ means that $A,B\subset\ZZ_M$ are both sets and $A(X)B(X)=\frac{X^M-1}{X-1}$ mod $X^M-1$.

The Chinese Remainder Theorem identifies the cyclic group $\ZZ_M$ with $\ZZ_{p_1^{n_1}}\oplus \ZZ_{p_2^{n_2}}\oplus\ldots \oplus\ZZ_{p_K^{n_K}}$. This allows us to interpret $\ZZ_M$ as a $K$-dimensional lattice, periodic in each direction, with
an explicit coordinate system defined as follows. For $j\in\{1,\dots,K\}$, let 
$M_j: = M/p_j^{n_j}= \prod_{i\neq j} p_i^{n_i}$. 
Then each $x\in \ZZ_M$ has a unique representation
$$
x=\sum_{j=1}^K x_j M_j,\ \ x_j\in \ZZ_{p_j^{n_j}}.
$$
For $D|M$, a {\em $D$-grid} in $\ZZ_M$ is a set of the form
$$
\Lambda(x,D):= x*D\ZZ_M=\{x'\in\ZZ_M:\ D|(x-x')\}
$$
for some $x\in\ZZ_M$. We note a few important special cases.

\begin{itemize}
\item A {\em line} through $x\in\ZZ_M$ in the $p_\nu$ direction is the set
$\ell_\nu(x):= \Lambda(x,M_\nu)$.
\item A {\em plane} through $x\in\ZZ_M$ perpendicular to the $p_\nu$ direction, on the scale $M_\nu p_\nu^{\alpha_\nu}$, is the set
$\Pi(x,p_\nu^{\alpha_\nu}):=\Lambda(x,p_\nu^{\alpha_\nu}).$

\item A {\em fiber in the $p_\nu$ direction} is a set of the form $x*F_\nu$, where $x\in\ZZ_M$ and
$$
F_\nu=\{0,M/p_\nu,2M/p_\nu,\dots,(p_\nu-1)M/p_\nu\}.
$$
Thus $x*F_\nu=\Lambda(x,M/p_\nu)$.  

\end{itemize}

A set $A\subset \ZZ_M$ is {\em fibered in the $p_j$ direction} if there is a subset $A'\subset A$ such that $A=A'*F_j$.
Since in this paper we only consider fibers at the top scale $M$, we use the terms ``fiber" and ``fibered" instead of ``$M$-fiber" and ``$M$-fibered" as in \cite{LaLo1}, \cite{LaLo2}, \cite{LL-even}.


\section{Divisor isometries and the box product}\label{product-dilations}


\subsection{Divisor isometries}\label{Diviso}

We continue to assume that $M=p_1^{n_1}\dots p_K^{n_K}$, where $p_1,\dots,p_K$ are distinct primes. 
The {\em divisor set} of $A\subset\ZZ_M$ is the set
$$
\Div(A):=\{(a-a',M):\ a,a'\in A\}
$$

Our analysis here is motivated in part by the fundamental theorems of Sands \cite{Sands} and Tijdeman \cite{Tij}, which we now state.

\begin{theorem}\label{thm-sands} {\bf (Divisor exclusion; Sands \cite{Sands})}
Let $A,B\subset \ZZ_M$. Then $A\oplus B=\ZZ_M$
if and only if $|A|\,|B|=M$ and 
$$
\Div(A) \cap \Div(B)=\{M\}.
$$
\end{theorem}

\begin{theorem}\label{thm-tijdeman} {\bf (Dilation invariance of tiling; Tijdeman \cite{Tij})}
Let $A\oplus B=\ZZ_M$, and let $r\in\ZZ_M$ satisfy $(r,|A|)=1$. Then $rA\oplus B=\ZZ_M$ is also a tiling, where $rA=\{ra: \ a\in A\}$.
\end{theorem}

\begin{definition}\label{divisor-isometries}
Let $\psi:\ZZ_M\to \ZZ_M$ be a mapping. We will say that $\psi$ is a {\em divisor isometry} if for every $x,x'\in\ZZ_M$ we have
$$
(\psi(x)-\psi(x'),M)=(x-x',M).
$$
\end{definition}

Examples of divisor isometries include translations $\tau_c(x)=x-c$ for any $c\in\ZZ_M$ and 
invertible dilations $\psi_r(x)= rx$ for any $r\in R$, where
$$
R:=\{r\in\ZZ_M:\ (r,M)=1\}.
$$ 
We note that the composition of any number of translations and dilations is always a linear transformation. An example of a divisor isometry that is not a linear mapping is provided by a plane exchange mapping, defined as follows.
Let $c,c'\in\ZZ_M$ satisfy $(c-c',M)=M_ip_i^{\alpha-1}$ for some $i\in\{1,\dots,K\}$ and $0< \alpha \leq n_i$. Then
$$
\exch(c,c',p_i^{\alpha})(x):=\begin{cases}
x+(c'-c) &\hbox{ if }x\in\Pi(c,p_i^{\alpha}),\\
x+(c-c') &\hbox{ if }x\in\Pi(c',p_i^{\alpha}),\\
x&\hbox{ if }x\not\in\Pi(c,p_i^{\alpha_i})\cup \Pi(c',p_i^{\alpha_i}).
\end{cases}
$$

Lemma \ref{isometry-lemma} below is an immediate consequence of Theorem \ref{thm-sands}.

\begin{lemma}\label{isometry-lemma}
Let  $\psi:\ZZ_M\to \ZZ_M$ be a divisor isometry. If $A\oplus B=\ZZ_M$ is a tiling, then so is $\psi(A)\oplus B=\ZZ_M$.
\end{lemma}

%
%
%

We use $\phi$ to denote the Euler totient function: if 
$n=\prod_{i=1}^L q_i^{\beta_i}$, where $q_1,\dots,q_L$ are distinct primes and $\beta_i\in\NN$, then
$
\phi(n)= \prod_{i=1}^L (q_i-1)q_i^{\beta_i-1}.
$

\begin{lemma}\label{basic-r} {\bf (Properties of dilations)}
For $r\in R$, let $\psi_r$ be the dilation $\psi_r(z)=rz$.
Let $x,x'\in\ZZ_M$ with $(x,M)=(x',M)=m$. Let
$$
R_{x,x'}:=\{r\in R:\ \psi_r(x)=x'\}.
$$
Then
$$
|R_{x,x'}|=\frac{\phi(M)}{\phi(M/m)}.
$$
Moreover, for any $r\in R_{x,x'}$ we have 
$$R_{x,x'}=\Lambda(r,M/m)\cap R.$$
\end{lemma}

\begin{proof}
%
Let $x,x'\in\ZZ_M$ with $(x,M)=(x',M)=m$. We first check that there exists at least one $r\in R$ such that $rx=x'$. Indeed, since $(x/m,M)=(x'/M)=1$, the numbers $x/m$ and $x'/m$ are invertible in $\ZZ_M$. Let $r=(x/m)^{-1}x'/m$, then $(r,M)=1$ and $rx=(x/m)^{-1}(x'/m)x=(x/m)^{-1}(x/m)x'=x'$.

With $r$ as above, we have $r'\in R_{x,x'}$ if and only if $r\in R$ and $rx=r'x=x'$, so that $(r-r')x=0$ in $\ZZ_M$. Since $(x,M)=m$, we need $r-r'$ to be divisible by $M/m$. In other words, $r'\in R\cap\Lambda(r,M/m)$.

We now count the number $Y$ of elements $r'\in\Lambda(r,M/m)$ such that $(r',M)=1$. Without loss of generality, we may assume that $M/m=p_1^{\alpha_1}\dots p_l^{\alpha_l}$ for some $l\leq K$ and $\alpha_1,\dots,\alpha_l\geq 1$. Then
$|\Lambda(r,M/m)|=m$, and
\begin{align*}
Y&=m\cdot\prod_{j=l+1}^K \frac{p_{j}-1}{p_{j}} ,\\
\phi(M/m)Y &= \frac{M}{m}\, \prod_{j=1}^l \frac{p_{j}-1}{p_{j}}\cdot m\, \prod_{j=l+1}^K \frac{p_{j}-1}{p_{j}}
= \phi(M)
\end{align*}
as claimed.
\end{proof}

\begin{lemma}\label{multitransitive} {\bf (Multitransitivity of dilations)}
For $\nu=1,\dots,K$, let $x_\nu,x'_\nu\in\ZZ_M$ with $(x_\nu,M)=(x'_\nu,M)=M/p_\nu^{\alpha_\nu}$, where $\alpha_1,\dots,\alpha_K\geq 1$. Then there exists $r\in R$ such that $rx_\nu=x'_\nu$ for $\nu=1,\dots,K$.
\end{lemma}

\begin{proof}
By Lemma \ref{basic-r}, for each $\nu$ there exists $r_\nu\in R$ such that $r_\nu x_\nu=x'_\nu$. Furthermore,
$$
R_{x_\nu,x'_\nu}=\Pi(r_\nu,p_\nu^{\alpha_\nu})\cap R.
$$
Let $r\in\bigcap_{\nu=1}^K \Pi(r_\nu,p_\nu^{\alpha_\nu})$. We claim that $r\in R$. Indeed, suppose that $p_\nu|r$ for some $\nu$. Since $p_\nu^{\alpha_\nu}|r-r_\nu$ and $\alpha_\nu\geq 1$, this implies that $p_\nu|r_\nu$, contradicting the fact that $r_\nu\in R$. This establishes the claim, and proves that $r$ satisfies the conclusion of the lemma.
\end{proof}

A brief discussion of the relationship between Theorem \ref{thm-sands}, Theorem \ref{thm-tijdeman}, and Lemma \ref{isometry-lemma} is in order. 
The special case of Tijdeman's theorem with $r\in R$ was first proved by Sands \cite[Theorem 2]{Sands}. Sands used this to prove Theorem \ref{thm-sands} \cite[Theorem 3]{Sands}. 
Lemma \ref{isometry-lemma} uses Theorem \ref{thm-sands} to extend \cite[Theorem 2]{Sands} to more general divisor isometries.

Theorem \ref{thm-tijdeman} extends Theorem \ref{thm-sands} in a different direction. Instead of assuming as in \cite[Theorem 2]{Sands} that $r\in R$, Tijdeman only assumes that $r$ is relatively prime to $|A|$. This is a significantly weaker assumption, since the mapping $x\to rx$ with $(r,|A|)=1$ need not be a divisor isometry if $|B|$ has prime factors that $|A|$ does not have. Coven and Meyerowitz \cite{CM} used the full strength of Theorem \ref{thm-tijdeman} to prove that if $|A|$ tiles the integers, then it must in fact tile a cyclic group $\ZZ_M$ for some $M$ with the same prime factors as $|A|$.


\subsection{A combinatorial interpretation of the box product}\label{box product revisited}

Following \cite{LaLo1, LaLo2} we define for 
$x\in\ZZ_M$
\begin{align*}
\bbA_m[x] & = \# \{a\in A:\ (x-a,M)=m \}.
\end{align*}
If $X\subset \ZZ_M$ and $x\in\ZZ_M$, we write
$\bbA_m[x|X] = \# \{a\in A\cap X: \ (x-a,N)=m\}$.

For $A,B\subset\ZZ_M$, we define their 
{\em box product}
$$
\langle \bbA[x], \bbB[y] \rangle = \sum_{m|M} \frac{1}{\phi(M/m)} \bbA_m[x] \bbB_m[y].
$$

\begin{theorem}\label{GLW-thm}
(\cite{LaLo1}; following \cite[Theorem 1]{GLW}) If
$A\oplus B=\ZZ_M$ is a tiling, then 
\begin{equation}\label{e-ortho2}
\langle \bbA[x], \bbB[y] \rangle =1\ \ \forall x,y\in\ZZ_M.
\end{equation}
\end{theorem}

The proof of Theorem \ref{GLW-thm} in \cite{GLW}, \cite{LaLo1} uses discrete harmonic analysis. Below, we present an alternative combinatorial proof, based on averaging over dilations $\psi_r$ with $r\in R$.

\begin{proof}
Suppose that $A\oplus B=\ZZ_M$. Without loss of generality, we may assume that $x=y=0$. By Theorem \ref{thm-tijdeman}, we have $rA\oplus B=\ZZ_M$ for all $r\in R$. 

We count the number of triples $(a,b,r)\in A\times B\times R$ such that $ra+b=0$. We have $|R|=\phi(M)$, and for every $r\in R$, it follows from Theorem \ref{thm-tijdeman} that there is exactly one pair $(a,b)\in A\times B$ such that $ra+b=0$. Thus the number of such triples is $\phi(M)$.

On the other hand, let $m|M$, and suppose that $b\in B$ satisfies $(b,M)=m$. Then $(-b,M)=m$, and by Lemma \ref{basic-r}, for every $a\in A$ with $(a,M)=m$ we have
$$
|\{r\in R:\ \psi_r(a)=-b\}|=\frac{\phi(M)}{\phi(M/m)}.
$$
Thus the number of triples as above is also equal to
$$
\sum_{m|M} \frac{\phi(M)}{\phi(M/m)} \bbA_m[0] \bbB_m[0].
$$
The identity follows.
\end{proof}


\subsection{Saturating sets}\label{satsets}

Let $A\oplus B= \ZZ_M$, and $x,y\in\ZZ_M$. We define
$$
A_{x,y}:=\{a\in A:\ (x-a,M)=(y-b,M) \hbox{ for some }b\in B\},
$$
and similarly for $B_{y,x}$ with $A$ and $B$ interchanged. 
These are the sets that saturate the box product in (\ref{e-ortho2}), in the sense that
$$
\langle \bbA[x| A_{x,y} ], \bbB[y| B_{y,x} ] \rangle =1.
$$
We also define
the {\em saturating set} for $x$ with respect to $A$, 
$$
A_{x}:=\{a\in A: (x-a,M)\in\Div(B)\} =\bigcup_{b\in B} A_{x,b},
$$
with $B_{y}$ defined similarly.

We note an alternative description of saturating sets, based on dilations and closely related to the alternative proof of Theorem \ref{GLW-thm} presented above. 

\begin{lemma}\label{satsets-to-dilations}
Let $A\oplus B=\ZZ_M$ be a tiling, and let $x,y\in\ZZ_M$, $a\in A$, $b\in B$. Then the following are equivalent:
\begin{itemize}
\item[(i)] $a\in A_{x,y}$ and $b\in B_{y,x}$,
\item[(ii)] there exists $r\in R$ such that $x-a=r(y-b)$,
\item[(iii)] there exists $r\in R$ such that $0=(a-x) + (-r)(b-y)$ in the tiling $\tau_x(A)\oplus (-r)\tau_y(B)=\ZZ_M$, where we use $\tau_x$ to denote translations as in Section \ref{Diviso}.
\end{itemize}
\end{lemma}

\begin{proof}
The equivalence between (i) and (ii) follows from Lemma \ref{basic-r}. Part (iii) is a reformulation of (ii).
\end{proof}


\section{Splitting}\label{splitting-section}


\subsection{Splitting for fibers} We continue to assume that $M=p_1^{n_1}\dots p_K^{n_K}$ and $A\oplus B=\ZZ_M$.

\begin{definition}
For a set $Z\subset \ZZ_M$, define
\begin{align*}
\Sigma_{A}(Z)&=\{a\in A:\ z=a+b\hbox{ for some }z\in Z,\ b\in B\},
\\
\Sigma_{B}(Z)&=\{b\in B:\ z=a+b\hbox{ for some }z\in Z,\ a\in A\}.
\end{align*}
\end{definition}

These are the sets of elements of $A$ and $B$ that tile $Z$.
Note that $\Sigma_{A}(Z)$ depends on both $A$ and $B$; whenever more than one tiling complement of $A$ is being considered,  we will identify the relevant tiling explicitly.

\begin{definition}\label{def-splitting}
Let $Z=x*F_i\subset \ZZ_M$ be a fiber in the $p_i$ direction. We will say that $Z$ {\em splits with parity $(A,B)$} if:
\begin{itemize}
\item[(i)] $p_i^{n_i}|a-a'$ for any $a,a'\in\Sigma_A(Z)$,
\item[(ii)] $p_i^{n_i-1}\parallel b-b'$ for any two distinct $b,b'\in \Sigma_B(Z)$. 
\end{itemize}
\end{definition}

An extension to lower scales is easy and could be added with only minimal additional effort, but will not be needed here.

\begin{lemma}\label{no-upgrades}{\bf (Splitting for fibers)}
Every fiber $Z$ splits with parity either $(A,B)$ or $(B,A)$. In particular, if $Z$ is a fiber in the $p_i$ direction, then for any $a\in \Sigma_A(Z)$ and $b\in \Sigma_B(Z)$, we have
$\Sigma_A(Z)\subset\Pi(a,p_i^{n_i-1})$ and $\Sigma_B(Z)\subset\Pi(b,p_i^{n_i-1})$.
\end{lemma}

\begin{proof}
Let $\{z_0,z_1,\dots,z_{p_i-1}\}\subset \ZZ_M$ be a fiber in the $p_i$ direction. 
Without loss of generality, we may assume that $z_\nu=\nu M/p_i$ for $\nu=0,1,\dots,p_i-1$. 

Let $a_\nu\in A,b_\nu\in B$ satisfy $a_\nu+b_\nu=z_\nu$. Then for $\nu\neq\nu'$,
$$
(a_\nu-a_{\nu'})+ (b_\nu-b_{\nu'})=(a_\nu 
+b_\nu)- (a_{\nu'}+b_{\nu'})=(\nu-\nu')M/p_i,
$$
so that $(a_\nu-a_{\nu'},M/p_i)= (b_\nu-b_{\nu'},M/p_i)$.
The only way to reconcile this with divisor exclusion is to have either
\begin{equation}\label{split-e10}
p_i^{n_i}|a_\nu-a_{\nu'}
\hbox{ and }p_i^{n_i-1}\parallel b_\nu-b_{\nu'}
\end{equation}
or the same with $A$ and $B$ interchanged.

Assume now that (\ref{split-e10}) holds for some $\nu,\nu'$. Then for any other $\nu''\in\{0,1,\dots,p_i-1\}$, we must have either $p_i^{n_i}\nmid b_\nu-b_{\nu''}$ or 
$p_i^{n_i}\nmid b_{\nu'}-b_{\nu''}$. In both cases, by (\ref{split-e10}) with either $\nu$ or $\nu'$ replaced by $\nu''$, it follows that $p_i^{n_i}|a_\nu-a_{\nu''}$ and $p_i^{n_i-1}\parallel b_\nu-b_{\nu'}$. This implies splitting with parity $(A,B)$.
Similarly, if (\ref{split-e10}) holds for some $\nu,\nu'$ with $A$ and $B$ interchanged, we get splitting with parity $(B,A)$ instead.
\end{proof}

\begin{lemma}{\bf (Translation invariance of splitting)} \label{translate-splitting}
Let $c\in \ZZ_M$
. If $z*F_i$ splits with parity $(A,B)$ for some $z\in\ZZ_M$, then $(z-c)*F_i$ splits with parity $(\tau_c(A),B)$. Similarly, if $z*F_i$ splits with parity $(B,A)$, then $(z-c)*F_i$ splits with parity $(B,\tau_c(A))$.
	
\end{lemma}

\begin{proof}
Given $u\in\ZZ_M$, we have $u\in z*F_i$ if and only if $u':= u-c\in (z-c)*F_i$, and $u=a+b$ with $a\in A,b\in B$ if and only if $u'=a'+b$ with $a'=a-c\in \tau_c(A)$ and $b\in B$. 
\end{proof}

Splitting is closely related to saturating sets. Our dilation-based reinterpretation of the latter shows them to be the sets of elements of $A$ and $B$ that tile a fixed element $z$ of $\ZZ_M$ in tilings $A\oplus rB$, where $r$ ranges over an appropriate set of dilations. Splitting does not take dilations into account; however, instead of tiling just one element of $\ZZ_M$, we tile an entire arithmetic progression at the same time. 

\begin{lemma}\label{divallign}{\bf (Splitting and saturating sets)}
	Suppose that $0\in A\cap B$. The following are equivalent:
	\begin{itemize}
		\item[(i)] 
		$a\in A_{x,y}$ and $b\in B_{y,x}$ for $x=\nu M/p_i$ and $y=\nu' M/p_i$
		(in other words,
	$(a-\nu M/p_i,M)=(b-\nu' M/p_i,M)$) for some $\nu,\nu'\in \{0,\ldots,p_i-1\}$,
		\item[(ii)] $a+rb=\nu'' M/p_i$ for some $r\in R$ and $\nu''\in\{0,\ldots,p_i-1\}$.
	\end{itemize}
\end{lemma}
\begin{proof}
	We have that (i) holds if and only if $a-\nu M/p_i=-r(b-\nu' M/p_i)$, which in turn is equivalent to $a+rb=(\nu-r\nu') M/p_i$
\end{proof}


\subsection{Splitting uniformity}\label{uniformity-on-subsets}



\begin{definition}\label{uniform-splitting}
(i) We say that the tiling $A\oplus B=\ZZ_M$ has {\em uniform $(A,B)$ splitting parity in the $p_i$ direction} if all fibers in the $p_i$ direction split with parity $(A,B)$. Uniform $(B,A)$ splitting parity is defined analogously. 

\medskip

(ii) We say that the tiling $A\oplus B=\ZZ_M$ has {\em $A$-uniform $(A,B)$ splitting parity in the $p_i$ direction} if all fibers $a*F_i$, with $a\in A$, split with parity $(A,B)$. $A$-uniform $(B,A)$ splitting parity is defined analogously as well as the two possible $B$-uniform splitting parities. 

\end{definition}

We now focus on the consequences of $A$-uniform splitting parity.

\begin{lemma}\label{disj_sig_B}
	Assume that $0\in B$, and that $a_0\neq a_1\in A$ satisfy 
	\begin{equation}\label{a_0a_1plane}
		p_i^{n_i}|a_0-a_1.
	\end{equation}
	Assume further that $a_\mu*F_i$ splits with parity $(B,A)$ for $\mu=0,1$. Then $\Sigma_{A}(a_0*F_i)$ and $ \Sigma_{A}(a_1*F_i)$ are disjoint.
\end{lemma}
\begin{proof}
Suppose that the conclusion is false. Then there exist $a\in \Sigma_{A}(a_0*F_i)\cap\Sigma_{A}(a_1*F_i)$ and $b_0,b_1 \in \Sigma_B(a_\mu*F_i)$ such that 
	\begin{align*}
		& a+b_0=a_0+\nu M/p_i\\
		& a+b_1=a_1+\lambda M/p_i
	\end{align*}
	for some $\nu,\lambda\in\{1,\ldots ,p_i-1\} $. By the assumptions of the lemma,  $a_0*F_i$, $a_1*F_i$ split with parity $(B,A)$ and $0\in \Sigma_{B}(a_\mu*F_i)$, thus $p_i^{n_i}|b_\mu$ for $\mu =0,1$. Together with (\ref{a_0a_1plane}), this shows that $\nu=\lambda$. Subtracting the two equations above, we get $a_0-a_1=b_0-b_1$. Since $a_0\neq a_1$, this contradicts divisor exclusion. 
\end{proof}

\begin{lemma}\label{localdist}{\bf (Local uniform splitting implies uniform distribution)}
Assume that $0\in B$, and that there exists $a_0\in A$ such that for every element $a\in A\cap \Pi(a_0,p_i^{n_i-1})$ the fiber $a*F_i$ splits with parity $(B,A)$. Then for every $\nu\in \{1,\ldots,p_i-1\}$,
\begin{equation}\label{distr}
|A\cap\Pi(a_0,p_i^{n_i})|=|A\cap\Pi(a_0+\nu M/p_i,p_i^{n_i})|.
\end{equation}
Consequently, $\Phi_{p_i^{n_i}}|(A\cap\Pi(a_0,p_i^{n_i-1}))$.
\end{lemma}

\begin{proof}
We start by proving that
\begin{equation}\label{localineq}
|A\cap\Pi(a_0,p_i^{n_i})|\leq|A\cap\Pi(a_0+\nu M/p_i,p_i^{n_i})| \text{ for all } \nu=1,\ldots,p_i-1.
\end{equation}	
Assume first that $A\cap\Pi(a_0,p_i^{n_i})=\{a_0\}$. By assumption $a_0+0=a_0$ and $a_0*F_i$ splits with parity $(B,A)$, thus 
$$
|A\cap\Pi(a_0+\nu M/p_i,p_i^{n_i})|\geq 1 \text{ for all } \nu=1,\ldots,p_i-1,
$$
as required. When $A\cap\Pi(a_0,p_i^{n_i})$ is not a singleton, we simply apply Lemma \ref{disj_sig_B} to all pairs $a*F_i,a'*F_i$ with $a,a'\in A\cap\Pi(a_0,p_i^{n_i})$. We get that $\Sigma_A(a*F_i)$ and $\Sigma_A(a'*F_i)$ are disjoint for all $a,a'\in A\cap\Pi(a_0,p_i^{n_i})$, hence (\ref{localineq}) follows.

In order to prove (\ref{distr}), it remains to prove that the converse of (\ref{localineq}) holds for all $\nu\in \{1,\ldots,p_i-1\}$. Fix such $\nu$, and let $a_\nu\in \Pi(a_0+\nu M/p_i,p_i^{n_i})$ (the existence of such an element is guaranteed by (\ref{localineq})). Since $\Pi(a_0,p_i^{n_i-1})= \Pi(a_\nu,p_i^{n_i-1})$, the assumptions of the lemma hold with $a_0$ replaced by $a_\nu$, and the converse inequality follows by applying (\ref{localineq}) to $a_\nu$.

Finally, we note a combinatorial interpretation of prime power cyclotomic divisibility. Suppose that $Y\subset \ZZ_M$ is a set, and that $p^\alpha|M$. Since 
$$\Phi_{p^\alpha}(X)=\Phi_p(X^{p^{\alpha-1}})= 1+X^{p^{\alpha-1}}+X^{2p^{\alpha-1}}+\dots +X^{(p-1)p^{\alpha-1}},
$$
it follows that $\Phi_{p^\alpha}|Y$ if and only if
\begin{equation}\label{cyclocombo}
|Y\cap\Pi(y,p^{\alpha})|= \frac{1}{p}|Y\cap\Pi(y,p^{\alpha-1})|\ \ \forall y\in Y.
\end{equation}
By (\ref{distr}), this holds for the set $A\cap\Pi(a_0,p_i^{n_i-1})$ with $p=p_i$ and $\alpha=n_i$. This proves the last statement in the lemma.
\end{proof}


\begin{corollary}\label{Aunif}{\bf ($A$-uniform splitting implies cyclotomic divisibility)}
Suppose that $0\in B$ and that the tiling $A\oplus B=\ZZ_M$ has $A$-uniform $(B,A)$ splitting parity in the $p_i$ direction. Then $\Phi_{p_i^{n_i}}|A$. 
\end{corollary}

\begin{proof}
By (\ref{cyclocombo}), $\Phi_{p_i^{n_i}}|A$ if and only if $\Phi_{p_i^{n_i}}|A\cap \Pi(a,p_i^{n_i-1})$ for every $a\in A$. This together with Lemma \ref{localdist} implies the corollary.
\end{proof}

We do not know whether the converse of Corollary \ref{Aunif} holds, i.e. whether $\Phi_{p_i^{n_i}}|A$ implies that the tiling $A\oplus B=\ZZ_M$ has $A$-uniform $(B,A)$ splitting parity in the $p_i$ direction.


\section{Tiling reductions}\label{tiling-reductions}


The slab reduction in Theorem \ref{subtile} and Corollary \ref{slab-reduction} was introduced in \cite{LaLo1} and used in the proof of (T2) for 3 prime factors in \cite{LaLo2}.

\begin{theorem}\label{subtile}\cite[Theorem 6.5]{LaLo1}
Let $M=p_1^{n_1}\dots p_K^{n_K}$.
Assume that $A\oplus B=\ZZ_M$, and that $\Phi_{p_i^{n_i}}|A$ for some $i\in\{1,\dots,K\}$.
Define
$$
A_{p_i}=\{a\in A:\ 0\leq\pi_i(a)\leq p_i^{n_i-1}-1\}.
$$
Then the following are equivalent:

\begin{enumerate}
\item [(i)]  For any translate $A'$ of $A$, we have $A'_{p_i}\oplus B=\ZZ_{M/p_i}$.
 
\item [(ii)] For every $m$ such that $p_i^{n_i}|m|M$, we have
	\begin{equation}\label{slabcond}
		m\in \Div(A) \Rightarrow m/p_i\notin \Div(B).
	\end{equation}

\item[(iii)] For every $d$ such that $p_i^{n_i}|d|M$, at least one of the following holds:
$$
\Phi_d|A,
$$
$$
\Phi_{d/p_i}\Phi_{d/p_i^2}\dots \Phi_{d/p_i^{n_i}}\mid B.
$$
\end{enumerate}
\end{theorem}

If a tiling $A\oplus B=\ZZ_M$ satisfies the conditions of Theorem \ref{subtile}, we can reduce proving (T2) for $A$ and $B$ to proving it for the tilings $A'_{p_i}\oplus B=\ZZ_{M/p_i}$ defined above \cite[Corollary 6.7]{LaLo1}.
Corollary \ref{slab-reduction} below follows by iterating this procedure until (T2) is known to hold.

\begin{corollary}\label{slab-reduction} {\bf (Slab reduction)} 
Let $M_0\in\NN$. Assume that for any $M|M_0$, and for any tiling $A\oplus B=\ZZ_M$, at least one of the following holds:
\begin{itemize}
\item[(i)] both $A$ and $B$ satisfy (T2),
\item[(ii)] $A$ and $B$ obey the conditions of Theorem \ref{subtile} for some $p_i|M$, 
\item[(iii)] the conditions of Theorem \ref{subtile} hold with $A$ and $B$ interchanged for some $p_i|M$. 
\end{itemize}
Then in any tiling $A_0\oplus B_0=\ZZ_{M_0}$, both $A_0$ and $B_0$ satisfy (T2).
\end{corollary}

This section aims to connect the slab reduction with uniform splitting parity (Definition \ref{uniform-splitting}) in a fixed direction. 
This relationship is made precise in Lemma \ref{splittingslab}. 
Additionally, Theorem \ref{splittingslab} below establishes a direct link between the slab reduction and \cite[Conjecture 9.4]{LaLo1}, while Corollary \ref{slabcor} (i) may be viewed as a step towards resolving Conjecture 9.2 in \cite{LaLo1}.

Clearly, if $(A,B)$ satisfy (\ref{slabcond}) then the same holds for $(A,rB)$ for all $r\in R$. Thus $(A,B)$ satisfy the slab reduction conditions if and only if $(A,rB)$ satisfy the slab reduction conditions for any $r\in R$.

\begin{definition}
Let $A\oplus B=\ZZ_M$.  
Given $a\in A$ and $b\in B$, 
we say that the product $\langle\bbA[a*F_i],\bbB[b]\rangle$ {\em splits with parity $(A,B)$} if $A_{x,b}\subset \Pi(a,p_i^{n_i})$ for all $x\in a*F_i$, and {\em with parity $(B,A)$} if $A_{x,b}\subset \Pi(x,p_i^{n_i})$ for all $x\in a*F_i$.
\end{definition}

\begin{lemma}\label{splittingslab}
Let $A\oplus B=\ZZ_M$. 
Then the following are equivalent:
\begin{enumerate}
\item [(I)] $(A,B)$ satisfy the slab reduction conditions in Theorem \ref{subtile}.
\item [(II)] The tiling $A\oplus rB=\ZZ_M$
has uniform $(rB,A)$ splitting parity in the $p_i$ direction, for every $r\in R$.
\item [(III)] For every $a\in A$ and $b\in B$, the product $\langle\bbA[a*F_i],\bbB[b]\rangle$ splits with parity $(B,A)$.

\end{enumerate}
\end{lemma}
\begin{proof}
We start by noting that the properties (I)--(III) are all translation-invariant. This is clear for (I) and (III); for (II), we use that $r(B-c)=rB-rc$ and that uniform splitting parity is invariant under translations. The same is true of the conditions in Theorem \ref{subtile}.

We first prove that (I) and (II) are equivalent.
Suppose that $(A,B)$ satisfy the slab reduction conditions. Note that the condition (ii) of Theorem \ref{subtile} depends only on the sets $\Div(A)$ and $\Div(B)$. Since $\Div(B)=\Div(rB)$ for $r\in R$, it follows that $(A,rB)$ must also satisfy the slab reduction conditions for such $r$.
Assume by contradiction that there exist $z\in\ZZ_M$ and $r\in R$ such that $z*F_i$ splits with parity $(A,rB)$ in the tiling $A\oplus rB=\ZZ_M$.
Let $a\in A,b\in B$ satisfy $z=a+rb$.
Without loss of generality, we may assume that $a\in A_{p_i}$. We prove that $(A_{p_i},rB)$ do not tile $\ZZ_{M/p_i}$. 

By the parity assumption, there exist 
$$
a_1,\ldots, a_{p_i-1}\in A\cap \Pi(a,p_i^{n_i})\subset A_{p_i}
$$
and 
$$
b_1,\ldots, b_{p_i-1}\in B\cap (\Pi(b,p_i^{n_i-1})\setminus\Pi(b,p_i^{n_i}))
$$ 
such that $a_\mu+rb_\mu=z+\mu M/p_i, \mu=1,\ldots, p_i-1$. It follows that $(a+rb,M/p_i)=(a_\mu+rb_\mu,M/p_i)$ for all $\mu\in \{1,\ldots,p_i-1\}$, which is a contradiction, since all sums $a+rb$ with $a\in A_{p_i}$ and $b\in B$ must be distinct modulo $M/p_i$.

For the other direction, assume that (II) holds. 
It follows from Corollary \ref{Aunif} that $\Phi_{p_i^{n_i}}|A$. 
Assume by contradiction that $(A,B)$ do not satisfy the slab reduction conditions (i)-(iii). Then 
(\ref{slabcond}) must fail, so that there exist $p_i^{n_i}|m|M$ such that $m\in \Div(A)$ and $m/p_i\in \Div(B)$. By translational invariance, we may assume that
\begin{equation}\label{zerocommon}
0\in A\cap B
\end{equation}
and that there exist $a\in A, b\in B$ and $\mu \in \{1,\ldots, p_i-1\}$ satisfying
$$
(a,M)=(b-\mu M/p_i,M)=m.
$$
By Lemma \ref{divallign}, there exists $r\in R$ such that $r(b-\mu M/p_i ) =-a$, so that $a+rb=r\mu M/p_i\in F_i$. By (\ref{zerocommon}), the latter implies that $F_i$ splits with parity $(A,rB)$, contradicting (II).

Next, we prove that (II)  and (III) are equivalent. We show that (II) implies (III), but the argument is completely reversible. Indeed, suppose that (II) holds yet (III) fails. Without loss of generality, we may assume that $0\in A\cap B$ and that (III) fails with $a=b=0$ and $x=M/p_i$, so that  $A_{M/p_i,0}\cap \Pi(0,p_i^{n_i})\neq \emptyset$. This implies that there exist $p_i^{n_i}|m|M$ and 
\begin{equation}\label{reverse}
	a\in A\cap \Pi(0,p_i^{n_i}), b\in B\cap \Pi(M/p_i,p_i^{n_i})
\end{equation}
satisfying 
$$
(a,M)=(b-M/p_i,M)=m.
$$
By Lemma \ref{divallign}, there exists $r\in R$ such that $r(b-M/p_i)=-a$, so that $a+rb=\mu M/p_i$ for some $\mu \in \{1,\ldots p_i-1\}$. By (\ref{reverse}), this means that the tiling pair $(A,rB)$ does not have uniform splitting in the $p_i$ direction with parity $(rB,A)$, contradicting (II).
\end{proof}

\begin{corollary}\label{slabcor}
Let $A\oplus B=\ZZ_M$ be a tiling. 
Assume that 
at least one of the following holds for some $i\in\{1,\dots,K\}$:
\begin{itemize}
\item[(i)] $\bbA_{M/p_i}[a]>0$ for every $a\in A$,
\item[(ii)] $\Phi_{p_i^{n_i}}|A$, and for every $b\in B$ we have 
\begin{equation}\label{maxplanebound}
|B\cap\Pi(b,p_i^{n_i})|=|B|/(|B|,p_i^{n_i}).
\end{equation}
\end{itemize}
Then $A$ satisfies the conditions of Theorem \ref{subtile}.
\end{corollary}

\begin{proof}
Suppose first that (i) holds. This clearly implies the condition (II) of Lemma \ref{splittingslab}, hence the conclusion follows from the lemma.

Assume now that (ii) holds. Since $\Phi_{p_i^{n_i}}|A$, we must have 
$$(|B|,p_i^{n_i})=\prod_{\alpha: 1\leq\alpha\leq n_i-1,\Phi_{p_i^\alpha}|B}\Phi_{p_i^\alpha}(1).$$
 It follows as in \cite[Lemma 4.3 and Corollary 4.4]{LaLo2} that for every $b\in B$, 
\begin{equation*}
|B\cap\Pi(b,p_i^{n_i-1})|\leq |B|/(|B|,p_i^{n_i}).
\end{equation*}
Combining this with (\ref{maxplanebound}), we see that for every $b\in B$ we must have 
$B\cap\Pi(b,p_i^{n_i-1})= B\cap\Pi(b,p_i^{n_i})$. This, again, clearly implies the condition (II) of Lemma \ref{splittingslab}.
\end{proof}


\section{Proof of Theorem \ref{T2-largeprime} and Corollary \ref{manyprimes}}\label{proof-largeprime}


\subsection{Proof of Theorem \ref{T2-largeprime}}\label{proof-thm}

Let $M=p_1^{n_1}\dots p_K^{n_K}$, where $p_1,\dots,p_K$ are distinct primes and $n_1,\dots,n_K\in\NN$.
Assume that $A\oplus B=\ZZ_M$. We first prove the theorem under the stronger assumption that
\begin{equation}\label{e-largeprime2}
p_1>M_1,
\end{equation}
where $M_1=p_2^{n_2}\dots p_K^{n_K}$.
We first claim that (\ref{e-largeprime2}) implies uniform splitting in the $p_1$ direction. Assume towards contradiction that $z*F_1$ splits with parity $(A,B)$ and $z'*F_1$ splits with parity $(B,A)$ for some $z,z'\in \ZZ_M$. We have $|\Sigma_B(z*F_1)|=p_1>M_1$, so that at least two elements of $\Sigma_B(z*F_1)$ lie on one line in the $p_1$ direction, with $M/p_1\in \Div(B)$. Similarly, the $(B,A)$ splitting parity implies that $M/p_1\in\Div(A)$. However, by divisor exclusion it is not possible for both of these to hold.

Assume therefore that the tiling has uniform $(B,A)$ splitting parity, with $M/p_1\in\Div(A)$. By the same argument as above, all tilings $A\oplus rB=\ZZ_M$ with $r\in R$ also have uniform $(B,A)$ splitting parity. 
It follows from Lemma \ref{splittingslab} that the tiling $A\oplus B=\ZZ_M$ satisfy the slab reduction conditions in Theorem \ref{subtile}.
By Corollary \ref{slab-reduction}, to prove that $A$ and $B$ satisfy (T2), it suffices to prove that (T2) holds for both sets in any tiling $A'\oplus B'=\ZZ_{M'}$, where $M'=M/p_1$. Iterating the procedure, we eventually reduce the question to proving (T2) for all tilings of period $M_1$. But at that point, the assumption on $M_1$ implies the conclusion.

For the full strength of the theorem, we need the proposition below.

\begin{proposition}\label{Blowbound}
Assume that 
\begin{equation}\label{Blowbound1}
p_i>(|B|,M_i).
\end{equation}
If $M/p_i\notin \Div(A)$, then the tiling $A\oplus B=\ZZ_M$ has uniform $(A,B)$ splitting parity in the $p_i$ direction.
\end{proposition}

\begin{proof}
Let $x\in \ZZ_M$. Suppose that $x*F_i$ splits with parity $(B,A)$, so that 
there exist $a_0,a_1,\ldots,a_{p_i-1}\in A$ and $b_0,b_1,\ldots,b_{p_i-1}\in B\cap\Pi(b_0,p_i^{n_i})$ such that 
$$
a_\nu+b_\nu =x+\nu M/p_i \text{ for }  \nu\in \{0,1,\ldots,p_i-1\}.
$$
We first claim that if $M/p_i\notin \Div(A)$, then
$|\Sigma_{B}(x*F_i)|=p_i$, and in particular 
\begin{equation}\label{Blowbound2}
|B\cap \Pi (b_0,p_i^{n_i})|\geq p_i.
\end{equation}
(Note that this part does not use (\ref{Blowbound1}).) Indeed, If we had $b_\nu= b_{\nu'}$ for some $\nu\neq\nu'$, it would follow that $a_\nu -a_{\nu'}=(\nu-\nu')M/p_i$, contradicting the assumption that $M/p_i\notin \Div(A)$. Hence the elements $b_0,b_1,\ldots,b_{p_i-1}$ are all distinct, proving the claim.

To complete the proof of the proposition, we recall the plane bound from \cite[Lemma 2.3]{LaLo1}: for any $z\in\ZZ_M$ we have
$$
|B\cap \Pi(z,p_i^{n_i})|\leq(|B|,M_i).
$$
This is clearly incompatible with (\ref{Blowbound2}) if (\ref{Blowbound1}) holds.
\end{proof}

We return to the proof of Theorem \ref{T2-largeprime}. Assume that (\ref{e-largeprime}) holds, so that
\begin{equation*}
p_1>p_2^{n_2-1}\dots p_K^{n_K-1}.
\end{equation*}
We proceed by induction in $M$. To initialize, if $M$ has at most 2 distinct prime factors, then the conclusion follows from the Coven-Meyerowitz theorem. Assume now that the theorem holds for all $\widetilde{M}|M$ with $\widetilde{M}\neq M$, and let $A\oplus B=\ZZ_M$ be a tiling. If there exists some $i\in \{1,\dots,K\}$ such that $p_i\nmid |A|$ or $p_i\nmid |B|$, we may apply the subgroup reduction in \cite[Lemma 2.5]{CM} (see also Theorem 6.1 and Corollary 6.2 in \cite{LaLo1}) to reduce proving (T2) for $A$ and $B$ to our inductive assumption on tilings of $\ZZ_{M/p_i}$.

We may therefore assume that $p_i$ divides both $|A|$ and $|B|$ for each $i$. This implies that
$$
(|A|,M_1)\leq p_2^{n_2-1}\dots p_K^{n_K-1}<p_1.
$$
and similarly for $|B|$. By divisor exclusion, at least one of $M/p_1\not\in\Div(A)$ and $M/p_1\not\in\Div(B)$ must hold. Assume without loss of generality that the latter holds. 
Applying Proposition \ref{Blowbound} with $i=1$, we see that the tiling has uniform $(B,A)$ splitting parity in the $p_1$ direction, and that $M/p_1\in\Div(A)$. The rest of the proof is as above: by Lemma \ref{splittingslab}, $A$ and $B$ satisfy the slab reduction conditions in Theorem \ref{subtile}. By Corollary \ref{slab-reduction}, it suffices to prove (T2) for both tiles in any tiling of period $\widetilde{M}=M/p_1$. But this again follows from our inductive assumption.


\subsection{Proof of Corollary \ref{manyprimes}}\label{proof-manyprimes}

Lemma \ref{almostsquarefree} is based on the same argument as \cite[Corollary 6.2]{LaLo1} (see also \cite{Tao-blog, dutkay-kraus, shi}), but we modify the statement slightly to fit the present context.

\begin{lemma}\label{almostsquarefree} 
Let $M=M_0 p_1^{n_1}p_2^{n_2}p_3^{n_3} \dots p_K^{n_K}$ and $p_1,\dots,p_K$ are distinct primes not dividing $M_0$. 
Assume that $ A\oplus B=\ZZ_M $, and that:
\begin{itemize}
\item For each $i\in\{1,\dots,K\}$, the prime factor $p_i$ divides at most one of $|A|$ and $|B|$.  
(This happens e.g., if $n_i=1$ for all $i$.)
\item For any tiling $A_0\oplus B_0=\ZZ_{M_0}$, both $A_0$ and $B_0$ satisfy (T2).
\end{itemize}
Then both $A$ and $B$ satisfy (T2). 
\end{lemma}

\begin{proof}
We proceed by induction in $n_1+\dots+n_K$. The case $K=0$ and $M=M_0$ is covered by the assumptions of the lemma. Suppose now that $K\geq 1$ and that 
(T2) holds for both $A'$ and $B'$ in any tiling $ A'\oplus B'=\ZZ_{M/p_K}$. By the assumption of the lemma, at least one of $|A|$ and $|B|$ is not divisible by $p_K$. Assume without loss of generality that $p_K\nmid |A|$. By Theorem \ref{thm-tijdeman}, $\tilde A\oplus B=\ZZ_M$ is again a tiling, where $\tilde A(X)=A(X^{p_K})$. Since $\tilde A\subset p_K\ZZ_M$, we may apply \cite[Theorem 6.1]{LaLo1} and the inductive assumption to conclude that $A$ and $B$ satisfy (T2). 
\end{proof}

We now prove Corollary \ref{manyprimes}.
Let
$$M^*=p_1^{\alpha_1}p_2^{\alpha_2}p_3^{\alpha_3}p_4p_5\dots p_K q_1^{\beta_1}\dots q_L^{\beta_L},
$$
where $p_1,\dots,p_K,q_1,\dots,q_L$ are distinct primes and the assumptions of Corollary \ref{manyprimes} hold. 

Each of the primes $p_4,\dots,p_K$, as well as $p_3$ if (\ref{alpha-e1}) holds, appears in the factorization of $M$ only once. Hence each of these primes may divide at most one of $|A|$ and $|B|$ in any tiling $A\oplus B=\ZZ_M$ with $M|M^*$. By Lemma \ref{almostsquarefree}, it suffices to prove the corollary with
$$
M^*=N_0 q_1^{\beta_1}\dots q_L^{\beta_L},
$$
which we now assume. 
Using (\ref{growth-e1}), and applying Theorem \ref{T2-largeprime} inductively, we can also remove all of the $q_j$ primes, starting with $q_L$ and working back to $q_1$. This leaves us with proving the corollary with
$$
M^*=N_0.
$$
Let $M|M^*$ and $A\oplus B=\ZZ_M$.
If $M=N_0 =p_1^{2}p_2^{2}p_3^{2}$ and $|A|=|B|=p_1p_2p_3$, then both $A$ and $B$ satisfy (T2) by Theorem \ref{T2-3primes-thm}. If only two of the primes $p_1,p_2,p_3$ divide $M$, then $A$ and $B$ satisfy (T2) by Theorem \ref{CM-thm}. In all other cases, there is an $i\in\{1,2,3\}$ such that $p_i$ divides at most one of $|A|$ and $B$. We then apply Lemma \ref{almostsquarefree} to pass from a tiling of period $M$ to tiling of period $M/p_i$. Continuing this procedure, we eventually get to the point where Theorem \ref{CM-thm} applies.


\section{Splitting consistency}\label{consistency}


Our work in Sections \ref{tiling-reductions} and \ref{proof-largeprime} shows that uniform splitting offers a viable path to proving the Coven-Meyerowitz conjecture. The goal of this section is to place the first few building blocks in the systematic study of uniform splitting. 

Partial results towards establishing uniform splitting might involve proving that splitting is uniform on certain subgrids of $\ZZ_M$. We are also interested in {\em splitting consistency} -- extensions of Lemma \ref{no-upgrades} saying that the elements of $A$ and $B$ that tile certain grids larger than a fiber must come from cosets of proper subgroups of $\ZZ_M$. In Section \ref{General-Section}, we prove such a result for 2-dimensional grids; a result of this type for 3-dimensional grids, but with additional assumptions, is given in Lemma \ref{consistent-splitting}.

In Section \ref{splitting-fibered}, we study tiling structure and splitting under the assumption that one of the sets $A$ and $B$ has fibered intersections with all $D(M)$-grids.
To provide some context, we describe briefly our approach to the case $M=p_i^{n_i}p_j^{n_j}p_k^{n_k}$. In \cite{LaLo2, LL-even}, given a tiling $A\oplus B=\ZZ_M$, we analyze the set whose mask polynomial is divisible by $\Phi_M$ (say, $A$). This divisibility condition places constraints on the structure of $A$ on any fixed $D(M)$-grid. Broadly speaking, on any fixed grid $A$ may either be written as a disjoint union of fibers in some fixed direction, or else it must have one out of several ``unfibered" structures we have identified explicitly. We say that $A$ is {\it fibered on $D(M)$-grids} if the former holds on every grid $\Lambda(a,D(M))$ for every $a\in A$. Our results in Section \ref{splitting-fibered} are an important part of the argument in the fibered grids case in \cite{LL-even}.

\subsection{General results}\label{General-Section}
We start with two results on splitting consistency on subgrids. In addition to applications in \cite{LL-even}, Lemma \ref{intersections} has been used in the proof of Lemma 6.4 in \cite{BL}.

\begin{lemma}\label{intersections}{\bf (Plane consistency)}
Let $M=p_1^{n_1}\dots p_K^{n_K}$, and assume that $A\oplus B=\ZZ_M$. Fix $i,j\in\{1,\dots,K\}$, and let
$\Lambda:= \Lambda(z,M/p_ip_j)$ for some $z\in\ZZ_M$.
Then there exists a $\nu\in\{i,j\}$ such that
$$
\Sigma_A(\Lambda)\subset A\cap \Pi(a,p_\nu^{n_\nu-1}),
$$
$$
\Sigma_B(\Lambda)\subset B\cap \Pi(b,p_\nu^{n_\nu-1}),
$$
where $a\in A$ and $b\in B$ satisfy $a+b=z$.

\end{lemma}

\begin{proof}
Assume towards contradiction that the lemma is false. Then there exist $z',z''\in \Lambda$ such that $z'=a'+b'$ and $z''=a''+b''$ with $a',a''\in A$ and $b',b''\in B$, and
\begin{equation}\label{split-e12}
p_i^{n_i-1}\nmid a-a',\ \ p_j^{n_j-1}\nmid b-b''.
\end{equation}
Let $z_{ij}$ be the unique element of $\Pi(z',p_j^{n_j})\cap \Pi(z'',p_i^{n_i})\cap\Lambda$, so that
$$
M/p_i \mid z_{ij}-z',\ \ M/p_j \mid z_{ij}-z''.
$$
Let $a_{ij}\in A$ and $b_{ij}\in B$ satisfy $z_{ij}=a_{ij}+b_{ij}$. By Lemma \ref{no-upgrades}, we have
$$
p_i^{n_i-1} \mid a_{ij}-a',\ \ p_j^{n_j-1} \mid b_{ij}-b''.
$$
Together with (\ref{split-e12}), this implies that
\begin{equation}\label{split-e13}
p_i^{n_i-1}\nmid a-a_{ij} \hbox{ and } p_j^{n_j-1}\nmid b-b_{ij}.
\end{equation}
However, we also have $(M/p_ip_j)|z-z_{ij}= (a-a_{ij})+(b-b_{ij})$. Hence $(a-a_{ij},p_k^{n_k})=(b-b_{ij},p_k^{n_k})$ for all $k\not\in\{i,j\}$, and, by (\ref{split-e13}), the same holds for $k\in\{i,j\}$.  This contradicts divisor exclusion and ends the proof of the lemma.
\end{proof}

\begin{corollary}\label{intersections-plus}{\bf (No cross-direction fibers)}
Let $M=p_1^{n_1}\dots p_K^{n_K}$, and assume that $A\oplus B=\ZZ_M$. Let $\Lambda_{ij}:= \Lambda(z,M/p_ip_j)$ as in Lemma \ref{intersections}. Suppose that $a_0*F_i\subset A$ for some $a_0 \in\Sigma_A(\Lambda_{ij})$. Then
\begin{equation}\label{planesigma-1dir}
\Sigma_A(\Lambda_{ij}) \subset \Pi(a_0,p_j^{n_j-1}).
\end{equation}
\end{corollary}

\begin{proof}
Let $b_0\in B$ satisfy $z_0:=a_0+b_0\in \Lambda_{ij}$. 
By translational invariance, we may assume that $z_0=a_0=b_0=0$, so that $F_i\subset A$.
Then $\Lambda_{ij}=\bigcup_{a\in F_i} a*F_j$. For each $a\in F_i$, we have $a=a+0$, with $a\in A$ and $0\in B$. Applying Lemma \ref{no-upgrades} to $a*F_j$, we see that $\Sigma_A(a*F_j) \subset \Pi(a_0,p_j^{n_j-1})$. But $a\in F_i$ was arbitrary, and (\ref{planesigma-1dir}) follows.
\end{proof}

\subsection{Splitting for fibered grids}\label{splitting-fibered}

In this section, we assume the following.

\medskip\noindent
{\bf Assumption (F):} We have $A\oplus B=\ZZ_M$, where $M=p_i^{n_i}p_j^{n_j} p_k^{n_k}$ has three distinct prime factors. Furthermore,
$\Phi_M|A$, and $A$ is fibered on $D(M)$-grids.

\medskip

Let $\cali$ be the set of elements of $A$ that belong to a fiber in the $p_i$ direction, that is, 
$$
\cali=\{a\in A:\ \bbA_{M/p_i}[a]=\phi(p_i)\}.
$$
The sets $\calj$ and $\calk$ are defined similarly.  
The assumption (F) implies that $A=\cali\cup\calj\cup\calk$. In general, this does {\em not} have to be a disjoint union and an element of $A$ may belong to two or all three of these sets.

We will write $D=D(M)$ for short.
We also define the direction assignment function $\kappa:A\to\{i,j,k\}$ as follows. Write $A=\bigcup_{x\in\{1,\dots,M/D\}} A\cap\Lambda(x,D)$. By (F), for each $x\in\{1,\dots,M/D\}$ such that $A\cap\Lambda(x,D)$ is nonempty, there exists $\nu(x)\in\{i,j,k\}$ such that $A\cap\Lambda(x,D)$ is fibered in the $p_{\nu(x)}$ direction. (If there exists more than one such $\nu$, we choose one arbitrarily and fix that choice.) We then let $\kappa(a):=\nu(x)$ for all $a\in A\cap\Lambda(x,D)$.

With this definition, we have $\kappa(a)=\kappa(a')$ for all $a,a'\in A$ such that $D|(a-a')$. We also define
$$F(a):= a*F_{\kappa(a)}\hbox{ for }a\in A.
$$
Observe that, for any $a,a'\in A$, the fibers $F(a)$ and $F(a')$ are either identical or disjoint. Indeed, suppose that $a''\in F(a)\cap F(a')$; then $a,a'$ belong to the same grid $\Lambda(a'',D)$, so that $\kappa(a)=\kappa(a')$ by definition, and $F(a)=F(a')$.

\begin{lemma}\label{fiberbasic}
Let $a,a'\in A$ and $b,b'\in B$. If 
$(b*F(a))\cap (b'*F(a'))\neq\emptyset$, then $b=b'$ and $F(a)=F(a')$.
\end{lemma}

\begin{proof}
Suppose that $z\in (b*F(a))\cap (b'*F(a'))$. Then $z=b+a_1=b'+a_2$ for some $a_1\in F(a)$ and $a_2\in F(a')$. It follows that $b=b'$ and $a_1=a_2$, so that $F(a)\cap F(a')\neq\emptyset$, implying that $F(a)=F(a')$ as claimed.
\end{proof}

We will prove that every $D$-grid is tiled by fibers in at most 2 directions, with the ``stratified" structure described in Lemma \ref{2dir}. We then prove a few results concerning the localization of $\Sigma_A(\Lambda)$, in the same spirit as Lemma \ref{intersections}, but with stronger conclusions since we now have the fibering assumption at our disposal.

The geometric intuition for the next result is provided by tilings of a 3-dimensional space  by rectangular columns. Define $R_i,R_j,R_k\subset \RR^3$  by
$$
R_i=\RR\times[0,1]\times[0,1],\ \ R_j=[0,1]\times\RR\times[0,1],\ \ R_k=[0,1]\times[0,1]\times \RR.
$$
Then any tiling of $\RR^3$ by translated, pairwise disjoint (up to sets of measure zero) copies of $R_i,R_j,R_k$ must use translated copies of at most two of these sets. Moreover, the tiling can be stratified into slabs of thickness 1, with each slab tiled by translates of just one of $R_i,R_j,R_k$.

\begin{lemma}\label{2dir}
Assume (F), and let $\Lambda=\Lambda(z_0,D)$ for some $z_0\in\ZZ_M$. 

\begin{itemize}

\item[(i)] Suppose that 
$$
\Sigma_A(\Lambda) \cap \cali\cap\calj\cap\calk\neq\emptyset,
$$
and that there exists $a\in\Sigma_A(\Lambda)\cap \cali\cap\calj\cap\calk$ such that $\kappa(a)=i$. Then $\kappa(a')=i$ for all $a'\in\Sigma_A(\Lambda)$, and in particular $\Sigma_A(\Lambda)\subset\cali$. 

\item[(ii)] Assume that
\begin{equation}\label{emptyint-2}
\Sigma_A(\Lambda) \cap \cali\cap\calj\cap\calk = \emptyset.
\end{equation} 
Then there is a subset $S\subset\{i,j,k\}$ with $|S|\leq 2$ such that $\kappa(a)\in S$ for all $a\in\Sigma_A(\Lambda)$. In particular, $\Sigma_A(\Lambda)$ is contained in the union of just two of the sets $\cali,\calj,\calk$.

\item[(iii)] Assume (possibly after a permutation of the indices $i,j,k$) that $S=\{i,j\}$, and let
$z_\nu=z_0+\nu M/p_k$ for $\nu=0,1,\dots,p_k-1$. Then for each $\nu$, there is a $\lambda(\nu)\in\{i,j\}$ such that
\begin{equation}\label{tiledbyFi}
\kappa(a)=\lambda(\nu) \hbox{ for all }   a\in \Sigma_A(z_\nu*F_i*F_j)).
\end{equation}

\end{itemize}
\end{lemma}

\begin{proof}
We will use throughout the proof that if $a\in \Sigma_A(\Lambda)$ and $a'\in\Lambda(a,D)$, then $F(a')\subset\Sigma_A(\Lambda)$. This is because if $a+b\in\Lambda$ for some $b\in B$, then we also have $b*F(a')\subset\Lambda$.

We first prove (i). Assume that $a\in\Sigma_A(\Lambda)\cap \cali\cap\calj\cap\calk$ such that $\kappa(a)=i$. Then 
$(a*F_j)\cup (a*F_k)\subset A$, and since $\kappa(a')=\kappa(a)$ for all $a'\in A\cap \Lambda(a,D)$, we have
$$
A_0:=(a*F_i*F_j)\cup (a*F_i*F_k)\subset A.
$$
Let $a''\in\Sigma_A(\Lambda)$. If $a''\in A\cap \Lambda(a,D)$, we already have $\kappa(a'')=i$. Assume now that $a''\in A\setminus \Lambda(a,D)$, and let $b,b''\in B$ satisfy $b*A_0\subset\Lambda$ and $b''*F(a'')\subset\Lambda$. By Lemma \ref{fiberbasic}, $b''*F(a'')$ must be disjoint from $b*A_0$. This is possible only if $\kappa(a'')=i$, proving (i).

We split the proof of (ii) and (iii) in two cases.

\medskip\noindent
{\bf Case 1.} Assume that (\ref{emptyint-2}) holds, and that some $a\in\Sigma_A(\Lambda)$ 
belongs to two of the sets $\cali,\calj,\calk$. Without loss of generality, we may assume that $a\in\cali\cap\calj$
and that $\kappa(a)=i$. Then 
$a*F_i*F_j\subset \Sigma_A(\Lambda)$ and $\kappa(a')=\kappa(a)=i$ for all $a'\in a*F_i*F_j$. We have $(a+b)*F_i*F_j=z_\nu*F_i*F_j$ for some $b\in B$ and $\nu\in\{0,1,\dots,p_k-1\}$, so that (\ref{tiledbyFi}) is proved with $\lambda(\nu)=i$ for that $\nu$.

Let $\mu\neq\nu$, and let $a_\mu\in A$ and $b_\mu\in B$ satisfy $a_\mu+b_\mu\in z_\mu*F_i*F_j$. By Lemma \ref{fiberbasic}, $b_\mu*F(a_\mu)$ must be disjoint from $(a+b)*F_i*F_j$, hence $\kappa(a_\mu)=\lambda$ for some $\lambda\in\{i,j\}$. Furthermore, if $a'_\mu\in A$ and $b'_\mu\in B$ satisfy $a'_\mu+b'_\mu\in z_\mu*F_i*F_j$, then 
$b'_\mu *F(a'_\mu)$ must be either identical to $b_\mu*F(a_\mu)$, or else it must be disjoint from both $b_\mu*F(a_\mu)$ and $(a+b)*F_i*F_j$. In both cases we have $\kappa(a'_\mu)=\kappa(a_\mu)=\lambda$. This proves (ii)--(iii) in Case 1.

\medskip\noindent
{\bf Case 2.} 
Assume now that for each $z\in\Lambda$, we have $z=a+b$, where $b\in B$ and $a$ belongs to exactly one of the sets $\cali,\calj,\calk$. 

We first prove that $\kappa(a)$ with $a\in\Sigma_A(\Lambda)$ may take at most two distinct values. We argue by contradiction.
Suppose that there exist $a_i,a_j,a_k\in A$ and $b_i,b_j,b_k\in B$ with $\kappa(a_\lambda)=\lambda$ and $x_\lambda:= a_\lambda+b_\lambda\in\Lambda$ for $\lambda=i,j,k$. Let $G_\lambda:=b_\lambda*F(a_\lambda)$. Let
$$ y\in\Pi(z_i,p_k^{n_k})\cap \Pi(z_j,p_i^{n_i})\cap \Pi(z_k,p_j^{n_j}),
$$
and suppose that $a\in A$ and $b\in B$ satisfy $a+b=y$. By Lemma \ref{fiberbasic}, the fiber $G:=b*F(a)$ must be disjoint from all of $G_i,G_j,G_k$. However, if $\kappa(a)=i$ then $G\cap G_k\neq\emptyset$; if $\kappa(a)=j$ then $G\cap G_i\neq\emptyset$; and if $\kappa(a)=k$ then $G\cap G_j\neq\emptyset$. There is no permitted value of $\kappa(a)$, contradicting our assumption. (We used a very similar argument in \cite[Proposition 7.10]{LaLo1}.)

Assume now that $\kappa(a)\in\{i,j\}$ with $a\in\Sigma_A(\Lambda)$. For each $\nu\in\{0,1,\dots,p_k-1\}$, let $\lambda(\nu)=\kappa(a_\nu)$, where $a_\nu$ satisfies $a_\nu+b_\nu=z_\nu$. Then for any $a'_\nu\in A$ and $b'_\nu\in B$ with $a'_\nu+b'_\nu\in z_\nu*F_i*F_j$, we must have $\kappa(a'_\nu)=\kappa(a_\nu)=\lambda_\nu$, since the other choice would contradict Lemma \ref{fiberbasic}.
\end{proof}

\begin{proposition}{\bf (Consistency for fibered grids)}\label{consistency3}
Assume that (F) holds, and that $0\in A\cap B$. Let $\Lambda:=\Lambda(0,D)$, and assume that $\kappa(0)=i$. Then:
\begin{itemize}
\item[(i)] There exists $l\in\{j,k\}$ such that
\begin{equation}\label{consistency-grid}
\Sigma_A(\Lambda)\subset \Pi(0,p_l^{n_l-1}).
\end{equation}
\item[(ii)] Assume further that $\Sigma_A(\Lambda)\subset\cali\cup\calj$, with both $\Sigma_A(\Lambda)\cap\cali\neq\emptyset$ and $\Sigma_A(\Lambda)\cap\calj\neq\emptyset$. Then (\ref{consistency-grid}) holds with $l=k$.
\end{itemize}
\end{proposition}

\begin{proof}
By Lemma \ref{2dir}, we may assume without loss of generality that $\kappa(a)\in\{i,j\}$ for all $a\in \Sigma_A(\Lambda)$. 
Let $\Lambda_{ij}=F_i*F_j$ and $\Lambda_{ik}=F_i*F_k$. 
For all $a'\in\Sigma_A(\Lambda_{ij})$, we must have $\kappa(a')=i$, since $\kappa(a')=j$ would contradict Lemma \ref{fiberbasic}. 


\medskip\noindent
{\bf Proof of (i).}
By Corollary \ref{intersections-plus}, we have
$$
\Sigma_A(\Lambda_{ij})\subset \Pi(0,p_j^{n_j-1}),\ \  \Sigma_A(\Lambda_{ik})\subset \Pi(0,p_k^{n_k-1}).
$$
It suffices to prove that, additionally, we have
\begin{equation}\label{consistency-grid2}
\text{ either } \Sigma_A(\Lambda_{ij})\subset \Pi(0,p_k^{n_k-1}) \text{ or } \Sigma_A(\Lambda_{ik})\subset \Pi(0,p_j^{n_j-1}).
\end{equation}
Indeed, suppose that the first inclusion in (\ref{consistency-grid2}) holds. Writing $\Lambda=\bigcup_{z\in \Lambda_{ij}} z*F_k$ and applying Lemma \ref{no-upgrades} to each $z*F_k$, we get that (\ref{consistency-grid}) holds with $l=k$. Assuming the second inclusion in (\ref{consistency-grid2}), we get the same with $l=j$.

It remains to prove (\ref{consistency-grid2}) 
Assume, by contradiction, that
$$
\Sigma_A(\Lambda_{ij})\not\subset \Pi(0,p_k^{n_k-1}) \text{ and } \Sigma_A(\Lambda_{ik})\not\subset \Pi(0,p_j^{n_j-1}).
$$
Then there must exist $z_j\in\Lambda_{ij}, z_k\in \Lambda_{ik}$, and $a_j,a_k\in A, b_j,b_k\in B$ satisfying
$$
a_\nu+b_\nu=z_\nu,\,\nu\in \{j,k\}.
$$
such that 
\begin{equation}\label{div_condition2}
p_j^{n_j-1}\nmid a_k, b_k,\ \ p_k^{n_k-1}\nmid a_j, b_j.
\end{equation}
Recall that $\kappa(a_j)=i$. Replacing $a_j$ by a different point of $a_j*F_i$ if necessary, we may assume that $z_j\in\Pi(z_k,p_i^{n_i})$. Let $z_{jk}\in\Lambda$ and $a_0\in F_i$ satisfy
$$
(z_{jk}-z_j,M)=(a_0-z_k,M)=M/p_k, \ \ (z_{jk}-z_k,M)=(a_0-z_j,M)= M/p_j.
$$
Let $a_{jk}\in A, b_{jk}\in B$ satisfy $z_{jk}=a_{jk}+b_{jk}$. By Lemma \ref{no-upgrades} and (\ref{div_condition2}),
$$
p_j^{n_j-1}\nmid a_k-a_{jk},\,p_k^{n_k-1}\nmid a_j-a_{jk}.
$$
On the other hand, since $z_{jk}\in a_0*F_j*F_k$, it follows from Lemma \ref{intersections} with $z=a_0$ that $p_\nu^{n_\nu-1}$ must divide $a_{jk}-a_0$ for some $\nu\in\{j,k\}$. This contradiction proves (i).

\medskip\noindent
{\bf Proof of (ii).}
Let $A_\nu=\{A\in\Sigma_A(\Lambda):\ \kappa(a)=\nu\}$ for $\nu=i,j$. 
Let $z_i,z_j\in\Lambda$ satisfy $z_i=a_i+b_i$, $z_j=a_j+b_j$ with $a_i\in A_i,a_j\in A_j, b_1,b_2\in B$. We claim that
\begin{equation}\label{div_condition}
p_k^{n_k-1}| a_i-a_j.
\end{equation}
Indeed, let $z'_i=a'_i+b_i\in b_i*F(a_i)$ and $z'_j= a'_j+b_j \in b_j*F(a_j)$ be the points such that $(z'_i-z'_j,M)=M/p_k$. 
By Lemma \ref{no-upgrades} applied to $z'_i*F_k$, we have $p_k^{n_k-1}| a'_i-a'_j$. But $a'_i\in F(a_i)$ and $a'_j\in F(a_j)$, hence (\ref{div_condition}) holds.

The proof of (ii) is completed as follows. 
Suppose that $A$ satisfies the assumptions of (ii). Then (\ref{div_condition}) implies the following: given $a_i\in A_i$, any $a_j\in A_j$ must satisfy $p_k^{n_k-1}|a_i-a_j$, and the same holds with the $i$ and $j$ indices interchanged. Since $A_1$ and $A_2$ are both nonempty, this clearly implies (\ref{consistency-grid}) with $l=k$.
\end{proof}

We briefly discuss splitting parity and the slab reduction for fibered grids. If $A\oplus B=\ZZ_M$, where $M$ has at most 3 prime factors, our expectation is that the condition of the slab reductions are satisfied for at least one of $A$ and $B$ in some direction \cite[Conjecture 9.1]{LaLo1}. By Lemma \ref{splittingslab}, this means uniform splitting parity in that direction for all tilings $A\oplus rB=\ZZ_M$ with $r\in R$. An intermediate result in this direction might establish uniform splitting parity on certain subgrids of $\ZZ_M$. For example, the following is easy to prove.

\begin{lemma}\label{consistent-splitting} 
Assume that (F) holds.
Let $\Lambda=\Lambda(z_0,D)$ for some $z_0\in\ZZ_M$. Assume that $\Sigma_A(\Lambda)$ has the structure described in Lemma \ref{2dir} (i), and that for each $\lambda\in\{i,j\}$,
\begin{equation}\label{2and2}
\#\{\nu\in\{0,1,\dots,p_k-1\}: \ \lambda(\nu)=\lambda\}\geq 2.
\end{equation}
Then the tiling has uniform splitting parity in the $p_k$ direction on $\Lambda$. (That is, all fibers $z*F_k$ with $z\in\Lambda$ split with the same parity, either $(A,B)$ for all $z$ or $(B,A)$ for all $z$).
\end{lemma}

\begin{proof}
We first claim the following. Let $z\in\Lambda$. Then the fibers $z'*F_k$ have the same splitting parity in the $p_k$ direction for all $z'\in z*F_i$.

To see this, assume that $\nu$ and $\mu$ are two distinct numbers in $\{0,1,\dots,p_k-1\}$ with $\lambda(\nu)=\lambda(\mu)=i$, and let $z\in \Lambda$. Let $y_\mu,y_\nu$ satisfy 
$$
(y_\mu-z,M)=(y_\nu-z,M)=M/p_k,\ y_\mu\in z_\mu*F_i*F_j,\ y_\nu\in z_\nu*F_i*F_j,
$$
with $z_\mu,z_\nu$ defined as in Lemma \ref{2dir}.  
Then $z_\nu=a_\nu+b_\nu$ and $z_\mu=a_\mu+b_\mu$, with $a_\nu,a_\mu\in A$, $b_\nu,b_\mu\in B$, and $\kappa(a_\nu)=\kappa(a_\mu)=i$.

Let $z'\in z*F_i$, so that $z'=z+\rho M/p_i$ for some $\rho \in\{0,1,\dots,p_i-1\}$. Consider the points $z'_\nu,z'_\mu\in x*F_k$ given by
$z'_\nu=z_\nu+\rho M/p_i$ and $z'_\mu=z_\mu+\rho M/p_i$. We have
$$
z'_\nu=(a_\nu+\rho M/p_i)+b_\nu,\ \  z'_\mu=(a_\mu+\rho M/p_i)+b_\mu.
$$
Thus $z*F_k$ splits with parity $(B,A)$ if and only if $p_k^{n_k}|b_\nu-b_\mu$, if and only if $z'*F_k$ splits with parity $(B,A)$. The case of parity $(A,B)$ is similar.

We now complete the proof of the lemma.
Let $z\in\Lambda$. By the claim above,  $z'*F_k$ splits with the same parity as $z*F_k$ for all $z'\in z*F_i$. By the same claim again, applied with $i$ and $j$ interchanged, $z''*F_k$ splits with the same parity as $z'*F_k$ for all $z''\in z'*F_j$. But this means that all fibers $z''*F_k$ with $z''\in z*F_i*F_j$ split with the same parity, as claimed.
\end{proof}

It is likely that, at least for $n_i=n_j=2$, Lemma \ref{consistent-splitting} can be proved using our current methods without the assumption that (\ref{2and2}) holds. However, since this extension is not needed in the proofs of our main results, we do not attempt it here.



\section{Acknowledgements}

The first author is supported by NSERC Discovery Grant 22R80520. The second author is supported by ISF Grant 607/21.



\bibliographystyle{amsplain}

\bigskip

\noindent{{\sc {\L}aba:} Department of Mathematics, University of British Columbia, Vancouver,
B.C. V6T 1Z2, Canada}, 
{\it ilaba@math.ubc.ca}

\noindent{ORCID ID: 0000-0002-3810-6121}

\medskip

\noindent{{\sc Londner:}  Department of Mathematics, Faculty of Mathematics and Computer Science, Weizmann Institute of Science, Rehovot 7610001, Israel},
{\it itay.londner@weizmann.ac.il}

\noindent{ORCID ID: 0000-0003-3337-9427}

\end{document}